\documentclass[10pt]{elsarticle}
\usepackage{hyperref}
\usepackage[latin9]{inputenc}
\usepackage{amsmath}
\usepackage{amsthm}
\usepackage{amssymb}
\usepackage{algpseudocode}
\usepackage{framed}
\usepackage{xcolor, colortbl}

\journal{Arxiv}

\newtheorem{definition}{Definition}
\newtheorem{theorem}{Theorem}
\newtheorem{proposition}{Proposition}
\newtheorem{lemma}{Lemma}
\newtheorem{corollary}{Corollary}

\begin{document}

\global\long\def\P{\mathbb{P}}%
\global\long\def\N{\mathbb{N}}%
\global\long\def\U{\mathbb{U}}%
\global\long\def\E{\mathbb{E}}%
\global\long\def\R{\mathbb{R}}%
\global\long\def\G{\mathcal{G}}%
\global\long\def\g{\mathbb{\Pi}}%
\global\long\def\F{\mathcal{F}}%
\global\long\def\S{\mathbb{S}}%
\global\long\def\Q{\mathcal{Q}}%
\global\long\def\B{\mathcal{B}}%
\global\long\def\ND{\mathcal{N}}%
\global\long\def\XX{\mathcal{X}}%
\global\long\def\indep#1{{\perp\hspace{-2mm}\perp}#1}%
\global\long\def\L{\mathcal{L}}%
\global\long\def\var{\mathrm{var}}%
\global\long\def\cov{\mathrm{cov}}%
\global\long\def\charf{\mathbf{1}}%
\global\long\def\d{\mathrm{d}}%
\global\long\def\M{\mathcal{M}}%
\global\long\def\Exp{\mathrm{Exp}}%
\global\long\def\Uniform{\mathrm{U}}%
\global\long\def\eqd{\stackrel{d}{=}}%
\global\long\def\eqas{\stackrel{a.s.}{=}}%
\global\long\def\X{\mathcal{X}}%
\global\long\def\supp{\mathrm{support}}%
\global\long\def\H{\mathcal{H}}%
\global\long\def\Z{\mathcal{Z}}%
\global\long\def\as{\qquad a.s.}%
\global\long\def\on{\qquad\text{on }}%
\global\long\def\C{\mathcal{C}}%
\global\long\def\barxi{\overline{\xi}}%
\global\long\def\Po{\mathrm{Po}}%
\global\long\def\Bi{\mathrm{Bi}}%
\global\long\def\Be{\mathrm{Be}}%
\global\long\def\defined{\stackrel{\text{def}}{=}}%
\global\long\def\barA{\overline{A}}%
\global\long\def\A{\mathcal{A}}%
\global\long\def\barx{x}%
\global\long\def\barX{TBD}%
\global\long\def\last{\ell}%
\global\long\def\bary{y}%
\global\long\def\barz{z}%
\global\long\def\DD{\mathcal{D}}%

\begin{frontmatter}

\title{Contractivity of Bellman Operator in Risk Averse Dynamic Programming with Infinite Horizon}

\author{Milo\v{s} Kopa}
\address{Charles University, Faculty of Mathematics and Physics, Department of Probability and Mathematical Statistics, Sokolovska 83, Prague, Czech Republic}

\author{Martin \v{S}m\'{i}d}
\address{The Czech Academy of Sciences, Institute of Information Theory and Automation, Pod Vodarenskou vezi 4, Prague, Czech Republic}



\begin{abstract}
The paper deals with a risk averse dynamic programming problem with infinite horizon. First, the required assumptions are formulated to have the problem well defined. Then the Bellman equation is derived, which may be also seen as a standalone reinforcement learning problem. The fact that the Bellman operator is contraction is proved, guaranteeing convergence of various solution algorithms used for dynamic programming as well as reinforcement learning problems, which we demonstrate on the value iteration algorithm.
\end{abstract}


\end{frontmatter}

\section{Introduction}
Risk averse variants of Dynamic Programming are widely studied. Our work is very close to \cite{ruszczynski2010risk} who, for a very similar setting, proves the convergence of the Value Iteration algorithm. The contribution of our work is two-fold. First, instead of complicated axiomatic definition, we defined the one-stage risk mapping constructively by means of its risk envelope, which is moreover independent of a choice of an underlying probability measure; consequently, the whole exposition, including proofs, is much simpler. Second, instead of the convergence of a particular solution algorithm, we prove the contractive property of the Bellman operator, which can be then plugged into convergence proofs of many different algorithms.    

Let $(\Omega, \F, P)$
be a probability space and let $\mathbf{F} := (\F_0, . . . ,\F_t,...)$
be a filtration, i.e. a sequence of increasing sigma algebras: $\F_0 \subset \F_1 \subset ... \subset \F$.
Consider a process
$\left\{Z_t\right\}, t = 0,1,...,$  adapted to the filtration $\mathbf{F}$, specifically $Z_t \in L_2(\F_t),  t = 0,1,...,$ 
.
For such a process in time $t$ we use coherent conditional risk measures: $\sigma_{t|\F_{t-1}}(Z_t)$, $t>0$. By saying that a conditional risk measure is coherent we mean it is \textcolor{black}{measurable,} monotonous ($\sigma_{t|\F_{t-1}}(X)\geq \sigma_{t|\F_{t-1}}(Y)$ for any random variables $X\geq Y$, $X,Y\in L_2$), \textcolor{black}{sub-additive ($\sigma_{t|\F_{t-1}}(X+Y) \leq \sigma_{t|\F_{t-1}}(X) + \sigma_{t|\F_{t-1}}(Y)$ for any random variables $X,Y \in L_2$)}
translation invariant ($\sigma_{t|\F_{t-1}}(X+C))=\sigma_{t|\F_{t-1}}(X)+C$ for any $C \in  L_2(\F_{t-1})$, $X\in L_2$) and positively homogeneous ($\sigma_{t|\F_{t-1}}(\Lambda X))=\Lambda \sigma_{t|\F_{t-1}}(X)$ for any $\Lambda \in L_2(\F_{t-1})$, $X \in L_2$).
Next we construct a nested risk measure $\rho_t$ as follows:
$$ \rho_t(Z_t) = \textcolor{black}{\sigma_{1|\F_{0}}(\sigma_{2|\F_{1}}(...\sigma_{t|\F_{t-1}}(Z_t)))}.$$ \textcolor{black}{It can be easily seen that, once $\F_0$ is trivial, $\rho_t(Z_t)$ is deterministic } a coherent risk measure. We will be interested in the limit version of this measure 
$\rho_{\infty}$ defined as:
\begin{equation}
\rho_{\infty}(Z) = \lim_{t\rightarrow \infty} \rho_t(Z_t) \qquad
\textcolor{black}{a.s.}
\label{rhoinf}
\end{equation}
Such a limit, however, may not exist in general and, therefore, we first formulate the sufficient conditions for the existence of $\rho_{\infty}$. 
\begin{definition}
\label{def1}
\textcolor{black}{We say that process $\left\{Z_t\right\}, t = 0,1,...,$ 
has uniformly bounded support if there exist finite $a < b$ such that  $\supp(Z_t) \subseteq \left\langle a,b \right\rangle$ for all $t$.}
\end{definition}
\begin{definition}
We say that conditional risk measure $\textcolor{black}{\sigma_{t|\F_{t-1}}}$  is  \textit{support-bounded} if for every $X \in L_2$ we have $\textcolor{black}{\sigma_{t|\F_{t-1}}(X)} \in \left\langle a,b \right\rangle$,  where $\textcolor{black}{a=\mathrm{essinf}(X)}$ and $\textcolor{black}{b=\mathrm{esssup}(X)}$.
\end{definition}

\begin{theorem}
Let process
$\left\{Z_t\right\}, t = 0,1,...,$  be adapted to the filtration $\mathbf{F}$, a.s. non-increasing, and let \textcolor{black}{the process have uniformly bounded support.}
Assume that a conditional risk measure $\sigma_{t|\F_{t-1}}(Z_t)$ is coherent and support-bounded for all $t$. Then $\rho_{\infty}(Z)$ exists.
\label{thm:exists}
\end{theorem}
\begin{proof}
\textcolor{black}{
First, thanks to the nested form of $\rho_t$ and the fact that conditional risk measures  $\sigma_{t|\F_{t-1}}$ are support-bounded, we have bounded sequence ${\rho_t(Z_t)}$, $t=0,1,...$. Second,
\begin{eqnarray*}
\rho_{t}(Z_{t}) - \rho_{t+1}(Z_{t+1}) &=&  \rho_t(Z_t) - \rho_t(\sigma_{t+1|\F_{t}}(Z_{t+1}))  \\
&\geq& \rho_t(Z_t) - \rho_t(\sigma_{t+1|\F_{t}}(Z_{t})) \\ &=& \rho_t(Z_t) - \rho_t(Z_t) = 0. 
\end{eqnarray*}
where the inequality follows from (i) coherency of $\sigma_{t+1|\F_{t}}$ and (ii) \textcolor{black}{the fact that $\left\{Z_t\right\}$ is } a.s. non-increasing, $t = 0,1,...,$. 
Hence, the sequence ${\rho_t(Z_t)}$, $t=0,1,...$ is non-increasing. Since every bounded non-increasing sequence has a limit, $\rho_{\infty}(Z)$ exists,  which completes the proof.}
\end{proof}
\noindent For the majority of practical situations, it suffices to assume $\Omega = [0,1] \times [0,1] \times \dots$ and $P=U(0,1)\otimes U(0,1)\otimes \dots$ where $U$ is uniform distribution; indeed, any process can be made Markov by adding its history to the state space and any coordinate of a Markov process can be expressed as a function of the past and an uniform variable (see \cite{kallenberg2001foundations}, Chp. 8).

It is well known (see \cite{ang2018dual}) that every coherent risk measure $\sigma$ can be expressed in a dual form: 
$\sigma(X)=\sup_{Q\in{\mathcal M}}\int X(\omega) Q(\omega)P(d\omega)$, where ${\mathcal M}=\{Q \in L_2: Q \geq 0, \E_P(Q)=1,\E_P(XQ)\leq \sigma(X),X\in L_2\}$  
is  a set of probability distributions known as risk envelope.
Especially, if $P=U(0,1)$, then $\sigma(X)=\sup_{Q\in{\mathcal M}'}\int_0^1 q_X(\omega) Q(\omega)d\omega$ where ${\mathcal M}'$ is a (different) risk envelope and $q_X$ is a quantile function of $X$.

Clearly, as conditional risk measures become coherent risk measures once the conditioning random element is fixed, it can be expressed by means of a dual representation too. Therefore, we further define $\sigma_{t|\F_{t-1}}$ by means of this representation; however, we do not allow the risk envelope to depend on $t$ and we do not allow it to be random: 
\begin{equation}
\sigma_{t|\F_{t-1}}(X)=\Sigma(\L(X|\F_{t-1})),\quad t\geq 1,
\qquad 
\Sigma(P)=\sup_{Q\in\mathcal{M}}\int_0^1
q_P(x) Q(x) dx,
\label{eq:homog}
\end{equation} 
where $\mathcal{M}$ is a deterministic risk envelope and $q_P$ is a quantile function corresponding to $P$. In practice this means that all the conditional measures are "of the same type", e.g. once $\sigma_{t|\F_{t-1}}$ is a conditional CVaR with risk level $\alpha$, then all the other conditional measures have to be conditional CVaRs with level $\alpha$, too.

Later we shall use the following Proposition.
\begin{proposition}\label{prop:homog}
Assume (\ref{eq:homog}). Let  $\F_0$ be trivial (implying that $Z_0$ is deterministic) and let 
\begin{equation}
\label{eq:eps}    
0 \leq Z_{t+1}-Z_t \leq \epsilon_t, \qquad t\geq 0,
\end{equation}
where $\epsilon_t$ is deterministic with $\sum_t \epsilon_t$ finite. Then $\rho_\infty$ exists and
$$
\rho_\infty(Z) = Z_0 + \sigma(\rho_\infty(Z')),
$$
where $Z'_t = Z_{t+1}-Z_0,t\geq 0$ and $\sigma(X) = \Sigma({\mathcal L}(X))$ ( see (\ref{eq:homog})).
\end{proposition} 
\noindent Note that $\sigma$ is unconditional coherent risk measure. First we prove the following Lemma:
\begin{lemma} (i) Let $Z_0$ be bounded and let  (\ref{eq:eps}) hold. Let $\sigma_{t|\F_{t-1}}(Z_t)$ be defined by (\ref{eq:homog}) and support-bounded for all $t$.
Then
$\rho_\infty(Z)$ exists and 
the convergence in (\ref{rhoinf}) is uniform in max norm.\\
(ii) Any coherent risk measure is continuous with respect to uniform convergence in max norm. 
\label{lem:cauchy}
\end{lemma}
\begin{proof}[Proof of Lemma] (i) Clearly, $Z_t$ fulfills the assumptions of Theorem \ref{thm:exists}, so $\rho_\infty(Z)$ exists. 
For any $t > 0$ and $s >0$, knowing that $\rho_t(Z_t)$ is non-decreasing, we have $$0 \leq \rho_{t+s}(Z_{t+s})-\rho_t(Z_t) \leq \rho_{t+s}(Z_t + e_t) - \rho_t(Z_{t}) = 
\rho_{t}(Z_t + e_t) - \rho_t(Z_{t}) = e_t 
$$
where $e_t=\sum_{\tau \geq t} \epsilon_\tau$.
\\
(ii) Let $Z_t \rightarrow Z^\star$ uniformly. Then there exists a sequence $e_t$ such that
$
Z^\star - e_t \leq Z_t \leq Z^\star + e_t$,
so, by coherence,
$\sigma(Z^\star) - e_t\leq \sigma(Z_t) \leq \sigma(Z^\star) + e_t
$ implying $\sigma(Z_t)\rightarrow \sigma(Z^\star)$.
\end{proof}

\begin{proof}[Proof of the Proposition] Thanks to Theorem \ref{thm:exists}, the limit defining the l.h.s. exists. From the definition and the coherence
$$
\rho_{\infty}(Z) = \lim_{t\rightarrow \infty} \rho_t(Z_t) = Z_0 + \lim_{t\rightarrow \infty} (\sigma(S_t))
$$
where $S_0=0$ and $$S_t = \sigma_{2|\F_1}(\dots\sigma_{t|\F_{t-1}}(Z'_{t-1})\dots)= \Sigma(\L(\Sigma(\dots\Sigma(\L(Z'_{t-1}|\F_{t-1})\dots)|\F_0)),$$ $t\geq 0$. By Lemma \ref{lem:cauchy} (i), $S_{t}$ converges uniformly to $\rho_\infty(Z')$, so, by (ii) of the same Lemma,
$
\lim_{t\rightarrow \infty} (\sigma(S_t))=\sigma(\lim_{t\rightarrow \infty}S_t)
=\sigma(\rho_\infty(Z')).
$
\end{proof}

\section{Contractiveness of the Bellman Operator}
Consider a dynamic programming problem
\[
V(s_0) := \sup_{a_t\in A(s_t), s_{t+1}=T(s_{t},a_{t},W_{t+1}),t\geq 0}\varrho_\infty(\sum_{t=0}^{\infty}\gamma^{t}r(s_{t},a_{t})),
\]
Here,
\begin{itemize}
\item $T : S \times A \times X \rightarrow S$ is a measurable mapping, where $S$ is a complete state space, $A$ is a (measurable) action space and $X$ is a measurable space
\item $W_{t} \in X$ is a Markov stochastic process (we may assume that it is i.i.d. uniform (see above)).
\item $r$ is a \textcolor{black}{uniformly bounded non-negative function}
\item $\varrho_\infty(Z) = -\rho_\infty(-Z)$, where $\rho_\infty$ is a limit nested risk measure (\ref{rhoinf}) defined by filtration $\F_{t}\defined\sigma(W_{t})$
and a support-bounded coherent risk measure $\sigma$ as in (\ref{eq:homog}).
\item $\gamma<1$ is a discount factor
\item $A(\bullet): S \rightarrow A$ is a set mapping.
\end{itemize}

\noindent Since $r$ is uniformly bounded non-negative and $\gamma<1$, process $-Z_T \defined \sum_{t=0}^{T}\gamma^{t}r_{t}(s_{t},a_{t}),$ $T = 0,1,...,$, has \textcolor{black}{uniformly bounded support and is non-increasing.} Combined with the assumption of coherent support-bounded conditional risk measure $\sigma$, it guarantees  existence of $\varrho_\infty$. Hence, the problem is well defined. 

\begin{proposition} (Bellman Equation)
\begin{equation}
V(s)=\sup_{a\in A(s)}[r(s,a)+\gamma\varsigma(V(T(s,a,W))],\qquad s \in S,
\label{eq:bellman}
\end{equation}
where $\varsigma(Z)=-\sigma(-Z)$.
\end{proposition}

\noindent Note that (\ref{eq:bellman}) may be also understood as a definition of a risk-averse version of  a reinforcement learning problem (see \cite{sutton2018reinforcement}).

\begin{proof}
Thanks to Proposition \ref{prop:homog}, basic properties of supremum, and Lemma \ref{lem:cauchy} (ii), we get, for any $s_0 \in S$: 
\begin{multline*}
V(s_0) = \sup_{a_t\in A(s_t),s_{t+1}=T(s_{t},a_{t},W_{t+1}),t\geq 0}
\left[r(s_0,a_0) + \gamma \varsigma\left(
\varrho_\infty\left(\sum_{t=1}^{\infty}\gamma^{t-1}r(s_{t},a_{t})\right)
\right)
\right]
\\
\sup_{a_0\in A(s_0)}
\left[r(s_0,a_0) + \gamma \varsigma\left(\sup_{a_t\in A(s_t),
s_{t+1}=T(s_{t},a_{t},W_{t+1}),t\geq 0}
\varrho_\infty\left(\sum_{t=1}^{\infty}\gamma^{t-1}r(s_{t},a_{t})\right)
\right)
\right],
\end{multline*}
which proves the Proposition.
\end{proof}

\begin{theorem}
The operator 
\[
B:(BV)(s)\defined\sup_{a\in A(s)}[r(s,a)+\gamma\varsigma(V(T(s,a,W))]
\]
is a $\gamma$ contraction w.r.t. sup norm.
\end{theorem}

\begin{proof}
Fix $\epsilon >0$ and, for any value function $V$, denote $a_{V,s}$ the $\epsilon$-optimal solution of $\sup_{a\in A(s)}[r(s,a)+\gamma\varsigma(V(T(s,a,W))]$. We have 
\begin{multline*}
\|BU-BV\|_{\infty}=\underbrace{\sup_{s\in S_{U}\defined\{s:(BU)(s)>(BV)(s))\}}[(BU)(s)-(BV)(s)]}_{b_{U}}\\
\vee\underbrace{\sup_{s\in S_{V}\defined\{s:(BU)(s)\leq(BV)(s))\}}[(BV)(s)-(BU)(s)]}_{b_{V}}
\end{multline*}
Further we have
\begin{multline*}
b_{U}=\sup_{s\in S_{U}}|\sup_{a\in A(s)}[r(s,a)+\gamma\varsigma(U(T(s,a,W))]-\sup_{a\in A(s)}[r(s,a)+\gamma\varsigma(V(T(s,a,W))]|\\
\leq\sup_{s\in S_{U}}[r(s,a_{U,s})+\gamma\varsigma(U(T(s,a_{U,s},W)))-r(s,a_{U,s})-\gamma\varsigma(V(T(s,a_{U,s},W)))] -  \epsilon\\
=\gamma\sup_{s\in S_{U}}[\varsigma(U(T(s,a_{U,s},W)))-\varsigma(V(T(s,a_{U,s},W)))]  -  \epsilon \\
=\gamma\sup_{s\in S_{U}}[-\sup_{Q\in\mathcal{M}}\int_0^1-U(T(s,a_{U,s},w))Q(w)dw+\sup_{Q\in\mathcal{M}}\int_0^1-V(T(s,a_{U,s},w))Q(w)dw] - \epsilon \\
\leq\gamma\sup_{s\in S_{U}}[-\int_0^1 -U(T(s,a_{U,s},w))Q_{V,s}(w)dw+\int_0^1-V(T(s,a_{U,s},w))Q_{V,s}(w)dw] - \epsilon \\
\leq\gamma\sup_{s\in S_{U}}\int_0^1|U(T(s,a_{U,s},w))-V(T(s,a_{U,s},w))|Q_{V,s}(w)dw - \epsilon \leq\gamma\|U-V\|_\infty - \epsilon
\end{multline*}
where $Q_{V,s}=\arg\max_{Q\in\mathcal{M}}\int_0^1-V(T(s,a_{U,s},w))Q(w)dw$
(the last inequality holds because $Q(w)dw$ is a probability measure). By releasing $\epsilon$ and performing a limit transition, we get the $b_U \leq \gamma\|U-V\|_\infty$. By making analogous steps for $b_V$, we get the Theorem.
\end{proof}

\noindent Thanks to these Theorem, many solution techniques, relying on the contractiveness of the Bellman operator, work when the expectation is replaced by a coherent risk measure. Here we only demonstrate this for the well known Value Iteration Algorithm (see \cite{sutton2018reinforcement})

Let $V_0: S \rightarrow [0,\infty)$ be arbitrary and let $\theta$ be a pre-chosen precision level. The Value Iteration Algorithm may be written as follows:
\begin{center}
\begin{minipage}{6cm}
\begin{framed}
\begin{algorithmic}
\State $n \leftarrow 0$
\Repeat
\State $n \leftarrow n+1$
\State $V_n \leftarrow B V_{n-1}$
\Until $\|V_n-V_{n-1}\|_\infty \leq \theta$
\end{algorithmic}
\end{framed}
\end{minipage}
\end{center}

\noindent The following result is a direct consequence of the Banach Fixed Point Theorem (see  \cite{granas2003fixed}, 1.1).
\begin{theorem} There exists $V_\star$ solving (\ref{eq:bellman}), 
$$
\| V_n - V_\star \|_\infty \leq \frac{\gamma^n}{1-\gamma } \| V_1 - V_ 0 \|_\infty
$$
for any $n$.
\end{theorem}
\begin{corollary}
The Value Iteration Algorithm stops after a finite number of steps.
\end{corollary}

\noindent {\bf Acknoledgements.} This work has been supported by grant No. GA19-11062S of the Czech Science Foundation.

\bibliography{bibliography}

\end{document}